\documentclass[12pt,reqno]{amsart}

\usepackage{amsmath, amsfonts, amssymb, amsthm, mathrsfs}

\usepackage[utf8,utf8x]{inputenc}
\usepackage[pagebackref=true]{hyperref}
\hypersetup{%
  pdftitle = {On Yuzvinsky's lattice sheaf cohomology for hyperplane arrangements},
  pdfkeywords = {},
  pdfsubject = {},
  pdfauthor = {},
  colorlinks = true,
  linktoc = page,
  citecolor = purple,
  linkcolor = blue 
}

\usepackage{xcolor}
\usepackage{caption}
\usepackage{tikz}
\usetikzlibrary{calc}
\usetikzlibrary{intersections}
\usetikzlibrary{cd}

\tikzcdset{scale cd/.style={every label/.append style={scale=#1},
cells={nodes={scale=#1}}}}

\usepackage{enumerate}
\usepackage{mfirstuc}

\newcommand{\ZZ }{\mathbb{Z}}
\newcommand{\KK }{\mathbb{K}}

\newcommand{\PP }{\mathbb{P}}

\newcommand{\id }{\mathrm{id}}

\newcommand{\Ac }{\mathcal{A}}

\newcommand{\Bc }{\mathcal{B}}

\newcommand{\Uc }{\mathcal{U}}

\newcommand{\mf}{\mathfrak m}

\newcommand{\Cs }{\mathscr{C}}
\newcommand{\Ds }{\mathscr{D}}
\newcommand{\Es }{\mathscr{E}}
\newcommand{\Fs }{\mathscr{F}}
\newcommand{\Gs }{\mathscr{G}}
\newcommand{\Os }{\mathscr{O}}

\newcommand{\isom }{\simeq}

\DeclareMathOperator{\Spec}{Spec}
\DeclareMathOperator{\Proj}{Proj}

\newcommand{\Xf }{\mathfrak{X}}

\DeclareMathOperator{\Der}{Der}

\DeclareMathOperator{\Tor }{Tor}

\DeclareMathOperator{\im }{im}
\DeclareMathOperator{\Tot }{Tot}

\DeclareMathOperator{\Ab }{\textbf{Ab}}
\DeclareMathOperator{\Mod}{\textbf{Mod} }
\DeclareMathOperator{\Sh}{\textbf{Sh} }
\DeclareMathOperator{\Rder }{R}

\DeclareMathOperator{\pd}{pd}

\newcommand\ddxi[1]{\partial/\partial x_{#1}}
\newcommand\wt[1]{\widetilde{#1}}

\DeclareMathOperator{\op }{op}

\numberwithin{equation}{section}

\theoremstyle{plain}
\newtheorem{lemma}[equation]{Lemma}
\newtheorem{theorem}[equation]{Theorem}

\newtheorem{conjecture}[equation]{Conjecture}

\newtheorem{corollary}[equation]{Corollary}
\newtheorem{proposition}[equation]{Proposition}
\theoremstyle{definition}
\newtheorem{definition}[equation]{Definition}
\newtheorem{remark}[equation]{Remark}

\newtheorem{problem}[equation]{Problem}

\title[On Yuzvinsky's lattice cohomology]
{On Yuzvinsky's lattice sheaf cohomology for hyperplane arrangements}

\author[P.~M\"ucksch]{Paul M\"ucksch}
\address{Paul M\"ucksch,
	Max-Planck-Institut f\"ur Mathematik,
	D-53111 Bonn, Germany}
\email{muecksch@mpim-bonn.mpg.de}

\begin{document}

\begin{abstract}
	We establish the relationship between 
	the cohomology of a certain sheaf on the intersection lattice
	of a hyperplane arrangement introduced by Yuzvinsky and the
	cohomology of the coherent sheaf on punctured affine space, respectively 
	projective space associated 
	to the module of logarithmic vector fields along the arrangement.
	Our main result gives a K\"unneth formula connecting the cohomology theories,
	answering a question by Yoshinaga.
	This, in turn, provides a characterization of the projective dimension of the 
	module of logarithmic vector fields
	and yields a new proof of Yuzvinsky's freeness criterion.
	Furthermore, our approach affords a new formulation of
	Terao's freeness conjecture and a more general problem.
\end{abstract}



\maketitle


\section{Introduction}
	\label{sec:Intro}

Let $\Ac$ be a hyperplane arrangement, i.e.\ a finite set of codimension one subspaces in a $\KK$-vector
space $V$ of dimension $\ell \geq 2$ for some field $\KK$.
The combinatorial structure of $\Ac$ is encoded in its intersection lattice $L(\Ac)$ which consists of 
all intersections of subsets of hyperplanes ordered by reverse inclusion.
Let $S = \KK[x_1,\ldots,x_\ell]$ be the coordinate ring of the vector space $V$.
The arrangement $\Ac$ is called free if the associated graded $S$-module $D(\Ac)$ of logarithmic vector fields along $\Ac $ or module of $\Ac$-derivations 
is a free $S$-module, a notion first introduced and studied by Saito \cite{Saito80_LogForms} and Terao \cite{Terao1980_FreeI}
(see Section \ref{ssect:Recoll_Arr}).
One of the most intricate problems in the study of hyperplane arrangements is to relate properties of $D(\Ac)$
to the combinatorial structure of $\Ac$ given by its intersection lattice.
The ultimate solution is proposed by Terao's conjecture from the 1980s (see \cite[Conj.~4.138]{OrTer92_Arr}) which asserts that over a fixed field $\KK$ the freeness of $\Ac$ only depends on its intersection lattice $L(\Ac)$.
This conjecture still remains open.

\bigskip

Functors defined on the intersection lattice of a hyperplane arrangement 
and related to the derivation module were already studied by Solomon
and Terao \cite{SolTer81_FormulaCharPoly}. They gave a new proof of Terao's seminal Factorization Theorem for free arrangements first obtained in \cite{Terao1981_GeneralizedExp}.

Assume that $\Ac$ is central and essential, that is $\cap_{H \in \Ac}H = \{0\}$ and 
set $L_0 := ( L(\Ac) \setminus \{ \{0\} \} )^{\op}$, i.e.\ the order relation in $L_0$ is inclusion.
In a series of papers \cite{Yuz91_LatticeCohom}, \cite{Yuz93_FirstTwoObstr}, \cite{Yuz93_FreeLocFreeL}
Yuzvinsky studied the functor $\Ds:L_0 \to \Mod_S$, 
$(X \subseteq Y) \mapsto (\Ds(X) = D(\Ac_X) \hookrightarrow D(\Ac_Y) = \Ds(Y))$
regarded as a sheaf on
the finite topological space associated to the poset $L_0$ and its cohomology (see Sections \ref{sec:PosetSheaves} and \ref{sec:LASheaves}).
An arrangement $\Ac$ is called locally free if all localization subarrangements $\Ac_X$ which consist of
all hyperplanes from $\Ac$ containing $X$ are free for all $X \in L_0$.
He showed \cite[Thm.~1.1]{Yuz93_FirstTwoObstr} that a locally free hyperplane arrangement $\Ac$
is free if and only if the lattice sheaf cohomology groups $H^n(L_0,\Ds)$ vanish for all $0<n<\ell-1$.

Moreover, in his study of these lattice sheaf cohomology groups, Yuzvinsky showed 
that free arrangements form a Zariski open subset 
in the moduli space of arrangements with a fixed intersection lattice, \cite[Cor.~3.4]{Yuz93_FreeLocFreeL}. 
This is up to date still the strongest general result towards Terao's conjecture.

\bigskip

A classical theorem by Horrocks \cite{Hor64_VectorBundles} asserts that a vector bundle $\Es$ 
on projective space $\PP^{\ell-1} = \Proj S$ 
splits into a direct sum of line bundles if and only if the sheaf cohomology groups
$H^n(\PP^{\ell-1},\Es(d))$ vanish for all $0<n<\ell-1$ and all $d \in \ZZ$.

It turns out that the coherent sheaf $\wt{D}$ on $\PP^{\ell-1}$ associated to the derivation module $D = D(\Ac)$
of a locally free arrangement
is a vector bundle, cf.\ \cite[Thm.~2.3]{MS01_LogFormsLocallyFree}.
Applying Horrocks' criterion to $\wt{D}$ of a locally free hyperplane arrangement
yields a freeness criterion resembling Yuzvinsky's criterion, cf.\ \cite[Prop.~1.20]{Yos14_FreenessSurv}.
A related similarity with local cohomology was already noticed by Yuzvinsky in \cite[Rem.~2.7]{Yuz91_LatticeCohom}.

\bigskip

Our aim is to establish
the exact relationship between Yuzvinsky's lattice sheaf cohomology 
and the sheaf cohomology on projective space and explain the resemblance of
Yuzvinsky's and Horrocks' criteria for freeness.
This clarifies the resemblance with local cohomology already noted by Yuzvinsky \cite[Rem.~2.7]{Yuz91_LatticeCohom} 
and answers a question posed by Yoshinaga \cite[Prob.~1.49]{Yos14_FreenessSurv}.

\bigskip

Set $\Xf := \Spec S \setminus \{ \mf \}$ where $\mf = (x_1,\ldots,x_\ell)$ is the homogeneous maximal ideal
and let $\Os_\Xf = \wt{S}|_\Xf$ be the structure sheaf (the restriction
of the structure sheaf of the affine scheme $\Spec S$ to the open complement $\Xf$ of the origin).

Our principal theorem establishes the exact relationship of the cohomology of the sheaf $\Ds$ on $L_0$ studied by Yuzvinsky
with the cohomology of the coherent sheaf $\wt{D}|_\Xf$ on the punctured spectrum $\Xf$
associated to the derivation module.

\begin{theorem}
	\label{thm:IsomCohom}
	For all $n \neq \ell-1$ we have
	\[
		H^n(\Xf,\wt{D}|_\Xf) \isom \bigoplus_{i+j=n} H^i(L_0,\Ds) \otimes_S H^j(\Xf,\Os_\Xf)
	\]
	and for $n = \ell-1$ we have a short exact sequence
	\begin{center}
		\begin{tikzcd}[column sep=5mm]
			0 
			\ar[r] & \bigoplus_{i+j={\ell-1}} H^i(L_0,\Ds) \otimes_S H^j(\Xf,\Os_\Xf)
			\ar[r] & H^{\ell-1}(\Xf,\wt{D}|_\Xf) \\
			\ar[r] & \Tor^S_1(H^1(L_0,\Ds),H^{\ell-1}(\Xf,\Os_\Xf))
			\ar[r] & 0.
		\end{tikzcd}
	\end{center}
	
	In particular, $H^n(\Xf,\wt{D}|_\Xf) \isom H^n(L_0,\Ds)$ for $n < \ell-1$.
\end{theorem}

Note that sheaf cohomology on the scheme $\Xf$ and
sheaf cohomology on projective space are connected as follows,
see e.g.\
\cite[{n\textsuperscript{o}} 69: Remarque]{Ser55_FAC}.
\begin{remark}
	\label{rem:ProjPunctSpec}
	Let $M$ be a finitely generated graded $S$-module.
	Denote by
	$\wt{M}|_\Xf$ the coherent sheaf associated to $M$ on $\Spec S$ restricted to the open subset 
	$\Xf = \Spec S \setminus \{ \mf \}$
	and by $\wt{M}$ the coherent sheaf on $\PP^{\ell-1} = \Proj S$ associated to $M$.
	Then
		$H^n(\Xf,\wt{M}|_\Xf) \isom \bigoplus_{d \in \ZZ} H^n(\PP^{\ell-1},\wt{M}(d))$
	for $n \geq 0$.
\end{remark}

As a direct consequence of Remark \ref{rem:ProjPunctSpec} and Theorem \ref{thm:IsomCohom} we obtain the following result
which establishes the relationship between the lattice sheaf cohomology studied by
Yuzvinsky and the sheaf cohomology on projective space.
This completely resolves a problem stated by Yoshinaga \cite[Prob.~1.49]{Yos14_FreenessSurv}
and readily yields another proof of Yuzvinsky's freeness criterion using Horrocks' theorem. 
\begin{theorem}
    \label{thm:YuzProj}
    For $n <\ell-1$ we have
    \[
        H^n(L_0,\Ds) \isom \bigoplus_{d \in \ZZ} H^n(\PP^{\ell-1},\wt{D}(d)).
    \]
\end{theorem}

\bigskip

The connection to local cohomology and projective dimension is as follows.
\begin{remark}
	\label{rem:LocalCohom}
	Recall that local cohomology is related to the cohomology on punctured affine space as follows, cf.\ \cite[Prop.~2.2]{Gro67_LocalCohom}.
	For $i>0$ we have:
	\[
		H^{i+1}_\mf(D) \isom H^{i}(\Xf,\wt{D}|_\Xf).
	\]
	
	Furthermore, by \cite[Thm.~3.8]{Gro67_LocalCohom} the depth, 
	respectively the projective dimension $\pd(D)$ (by the Auslander-Buchsbaum formula)
	of the module $D$ is tied to local cohomology by
	\[
		\pd(D) \leq p \,\text{ if and only if }\, H^i_\mf(D) = 0 \text{ for } i < \ell-p. 
	\]
\end{remark}

The module $D$ is reflexive (cf.\ \cite[p.~268]{Saito80_LogForms}) and as such, it
is the dual of another finitely generated module. So $D$ has projective dimension at most $\ell-2$.
Consequently, Theorem \ref{thm:IsomCohom} together with the preceding remark
directly yields the following characterization of the projective dimension of $D$.

\begin{theorem}
	\label{thm:pdD}
	The following two conditions are equivalent:
	\begin{itemize}
		\item[i)]
		$\pd(D) \leq p$;
		
		\item[ii)]
		$H^n(L_0,\Ds) = 0\, \text{ for } \, 0<n<\ell-1-p$.
	\end{itemize}
\end{theorem}  

Finally, as a consequence to Theorem \ref{thm:pdD}, we obtain the following stronger form of 
Yuzvinsky's freeness criterion \cite[Thm.~1.1]{Yuz93_FirstTwoObstr}
showing the assumption of $\Ac$ being locally free to be superfluous.

\begin{corollary}
	\label{cor:CharFreeness}
	The arrangement $\Ac$ is free if and only if
	\[
	H^n(L_0,\Ds) = 0 \quad \text{ for } \quad 0<n<\ell-1.
	\]
\end{corollary}

\bigskip

This paper is organized as follows.
In Section \ref{sect:Recoll} we review some basic notions from the theory of hyperplane arrangements.
Furthermore, we recall some
results from homological algebra and sheaf theory.
In Section \ref{sec:PosetSheaves} we review sheaves on posets and their cohomology.
Section \ref{sec:LASheaves} gives further details about some special sheaves on the intersection lattice
of an arrangement. 
In Section \ref{sec:Proofs} we prove Theorem \ref{thm:IsomCohom} and finally, in Section \ref{sec:ConclRmks}
we comment on possible generalizations of our approach and related problems.


\section{Recollection and preliminaries}
\label{sect:Recoll}

In this note $V \isom \KK^\ell$ always denotes an $\ell$-dimensional $\KK$-vector space
over some field $\KK$ where $\ell \geq 2$.

\bigskip

Let $S = \KK[x_1,\ldots,x_\ell]$ be the coordinate ring of $V$.
The ring $S$ is graded: $S = \bigoplus_{p\in \ZZ} S_p$ where 
$S_p$ is the $\KK$-space of homogeneous polynomials of degree $p$ (along with $0$) and $S_p = \{0\}$ for $p < 0$.

If $f \in S$ then we write $S_f = S[\frac{1}{f}]$ for the localization of $S$ by $f$ and similarly for an
$S$-module $M$ we write $M_f = M \otimes_S S_f$.


\subsection{Hyperplane arrangements}
\label{ssect:Recoll_Arr}

As a general reference for hyperplane arrangements we refer to the book by Orlik and Terao  \cite{OrTer92_Arr}.

Let $\Ac = (\Ac, V)$ be a hyperplane arrangement in $V$, that is a finite set of codimension one
subspaces of $V$.
The \emph{intersection lattice} of $\Ac$ is 
\[
    L(\Ac) = \left\{ \cap_{H \in \Bc} H \mid \Bc \subseteq \Ac \right\}
\]
with the partial order
\[
    X \leq Y \,:\iff\, X\supseteq Y \quad (X,Y \in L(\Ac)).
\]

In this note we always assume $\Ac$ to be \emph{essential}, that is 
for the maximal element $T(\Ac) := \cap_{H \in \Ac} H$ in $L(\Ac)$ we have $T(\Ac) = \{0\}$.

\bigskip

For $X \in L(\Ac)$ the \emph{localization $\Ac_X$} of $\Ac$ at $X$ is
\[
\Ac_X := \{H \in \Ac  \mid H \supseteq X \}.
\]

If $X,Y \in L(\Ac)$ then $X\wedge Y := \sup \{Z \in L(\Ac) \mid Z \leq X$ and $Z \leq Y \}$.
Note that we have $\Ac_{X\wedge Y} = \Ac_X \cap \Ac_Y$.

\bigskip

For all hyperplanes $H \in \Ac$ we fix $\alpha_H \in V^*$ with $H = \ker(\alpha_H)$.
The \emph{defining polynomial} $Q(\Ac)$ of $\Ac$ is
\[
    Q(\Ac) := \prod_{H \in \Ac} \alpha_H.
\]

A $\KK$-linear map $\theta:S\to S$ which satisfies $\theta(fg) = \theta(f)g + f\theta(g)$ is called a $\KK$-\emph{derivation}.
Let $\Der_\KK(S)$ be the $S$-module of $\KK$-derivations of $S$. 
It is a free $S$-module with basis $\ddxi{1},\ldots,\ddxi{\ell}$.

\begin{definition}
    \label{def:DerMod}
    The \emph{module of $\Ac$-derivations} is the $S$-submodule of $\Der_\KK(S)$ defined by
    \begin{equation*}
      D(\Ac) := \{ \theta \in \Der_\KK(S) \mid \theta(\alpha_H) \in \alpha_H S \text{ for all } H \in \Ac\}.
    \end{equation*}
    In particular, if $\Bc \subseteq \Ac$, then $D(\Ac) \subseteq D(\Bc)$.

    We say that $\Ac$ is \emph{free} if the module of $\Ac$-derivations is a free $S$-module. 
\end{definition}

\begin{definition}
    \label{def:QXpolys}    
    For $X \in L(\Ac)$ we define
    \[
        Q(X) := \prod_{H \in \Ac \setminus \Ac_X} \alpha_H 
            = \frac{Q(\Ac)}{Q(\Ac_X)}.
    \]
\end{definition}

The following observation provides a crucial ingredient in the proof of Theorem \ref{thm:IsomCohom}.

\begin{lemma}
    \label{lem:LocD}
    For all $X,Y \in L(\Ac)$ we have:
    \[
        D(\Ac_Y)_{Q(X)} = D(\Ac_{X\wedge Y})_{Q(X)}.
    \]    
\end{lemma}
\begin{proof}
    For each $H \in \Ac$ we define the $S$-module homomorphism $M_H:S^\ell \to S/ \alpha_H S$ by 
    \[
        M_H(f_1,\ldots,f_\ell) := \sum_{i=1}^\ell f_i \frac{\partial\alpha_H}{\partial x_i} + \alpha_H S. 
    \]
    For $X \in L(\Ac)$ we set 
    \[
        M_X:=\sum_{H \in \Ac_X} M_H:S^\ell \to \bigoplus_{H \in \Ac_X} S/ {\alpha_H S}.
    \]

    From the definition of $D(\Ac_Y)$ we have the following short exact sequence
    
    \begin{center}
    \begin{tikzcd}
        0 
        \ar[r] & D(\Ac_Y) 
        \ar[r] & S^\ell 
        \arrow{r}{M_Y} & \bigoplus_{H \in \Ac_Y} S/ {\alpha_H S} 
        \ar[r] & 0.
    \end{tikzcd}
    \end{center}

    If we localize at $Q(X)$, for each $H \in \Ac \setminus \Ac_X$ we have
    \begin{align*}
        S/ {\alpha_H S} \otimes_S S_{Q(X)} &= \, 0, \\
        M_H \otimes_S \id_{S_{Q(X)}} &\equiv \, 0.
    \end{align*}
    Recall, that $\Ac_{X\wedge Y} = \Ac_X \cap \Ac_Y$, thus
    \[
        \bigoplus_{H \in \Ac_Y} S / {\alpha_H S} \otimes_S S_{Q(X)} = \bigoplus_{H \in \Ac_{X\wedge Y}} S / {\alpha_H S} \otimes_S S_{Q(X)}.
    \]
    Further, recall that there is a natural inclusion
    $i:D(\Ac_Y) \hookrightarrow D(\Ac_{X\wedge Y})$.

    Since localization at $Q(X)$ is an exact functor, we obtain the following commutative diagram with
    exact rows:
    
    \begin{center}
    \begin{tikzcd}[column sep=5mm, scale cd=0.9]
        0 
        \ar[r] \ar[d, equal] & D(\Ac_Y)_{Q(X)}
        \ar[r] \ar[d, hook, "i\otimes_S \id"] & S_{Q(X)}^\ell
        \arrow{r}{M_Y \otimes_S \id } \ar[d, equal] &[0.8cm] \bigoplus_{H \in \Ac_Y} S/ {\alpha_H S}  \otimes_S S_{Q(X)}
        \ar[r] \ar[d, equal] & 0 \ar[d, equal] \\
        0 
        \ar[r] & D(\Ac_{X\wedge Y} )_{Q(X)}
        \ar[r] & S_{Q(X)}^\ell 
        \arrow{r}{ M_{X\wedge Y} \otimes_S \id } &[0.8cm] \bigoplus_{H \in \Ac_{X\wedge Y} } S/ {\alpha_H S}  \otimes_S S_{Q(X)}
        \ar[r] & 0 \\
    \end{tikzcd}
    \end{center}
    Hence $i\otimes_S \id$ yields the equality (e.g.\ by the five-lemma and extending the diagram by additional zeros to the left).
\end{proof}

We record the following special case of Lemma \ref{lem:LocD}.

\begin{corollary}
	Let $X \in L(\Ac)$. Then we have
    \label{coro:LocD}
    \[
        D(\Ac_X)_{Q(X)} = D(\Ac)_{Q(X)}.
    \]
\end{corollary}
\begin{proof}
    Let $Y=T(\Ac)$ in Lemma \ref{lem:LocD} and note that then $X\wedge Y = X$.
\end{proof}

\begin{remark}
	The preceding lemma and corollary can also be seen as a consequence of the local property of the functor or sheaf $\Ds$ \cite[Prop.~6.6]{SolTer81_FormulaCharPoly}: for a generic point $p \in X$ we have $D(\Ac_X)_p \isom D(\Ac)_p$.
\end{remark}


\subsection{Homological algebra}
\label{ssect:Recoll_HomAlg}

For the basics we refer to \cite{Rot09_IntroHomAlg}.
Let 

\begin{center}
\begin{tikzcd}[column sep=8mm]
    C^\bullet \quad = \quad \cdots 
    \ar[r] & C^{n-1} 
    \ar{r}{d^{n-1}} &C^n 
    \ar{r}{d^n} &C^{n+1} 
    \ar[r] &\cdots
\end{tikzcd}
\end{center}
be a cochain complex of abelian groups ($S$-modules). 
Then we write
\[
    H^n(C^\bullet) 
        = Z^n / B^n,
\]
for the $n$-th cohomology group (module),
where $Z^n = \ker(d^n)$ is the group ($S$-module) of $n$-cocycles and 
$B^n=\im(d^{n-1})$ is the group ($S$-module) of $n$-coboundaries.
Note that with the zero-coboundary maps
\begin{tikzcd}[column sep=14mm]
	B^n \ar{r}{d^n|_{B^n} \equiv 0} &B^{n+1} 
\end{tikzcd}
and the inclusion maps $i^n: B^n \hookrightarrow C^n$ the complex $B^\bullet$ 
is a subcomplex of $C^\bullet$ called the \emph{coboundary-subcomplex}.

Let $A^\bullet$ and $C^\bullet$ be two cochain complexes of $S$-modules with coboundary maps $d_A$ and $d_C$ respectively.
By $A^\bullet \otimes_S C^\bullet$ we denote their tensor product which is defined as
the total complex of the associated bicomplex, i.e.\
\[
    (A^\bullet \otimes_S C^\bullet)^n := \bigoplus_{i+j=n} A^i \otimes_S C^j
\] 
with coboundary maps 
\[
	d^n(a \otimes_S c) = d^i_A(a) \otimes_S c + (-1)^i a \otimes_S d_C^j(c)
\]
for $a \in A^i$, $c \in C^j$ and $i+j=n$.

To later guarantee the exactness of a tensor product of 
two special resolutions of sheaves, 
we require the following special case of the acyclic assembly lemma, cf.\ \cite[Lem.~2.7.3]{Wei94_HomAlg}.
\begin{lemma}
    \label{lem:AcycAss}
    Let $C^{\bullet,\bullet}$ be a bounded first quadrant bicomplex in an abelian category.
    If $C^{\bullet,\bullet}$ has exact rows or columns then $\Tot(C^{\bullet,\bullet})^\bullet$
    which is given by
    \[
        \Tot(C^{\bullet,\bullet})^n = \bigoplus_{i+j=n}C^{i,j}
    \]
    is also exact.
\end{lemma}

One crucial ingredient for the proof of our main Theorem \ref{thm:IsomCohom} is the following K\"unneth formula 
for the cohomology of the tensor product of two complexes.

\begin{theorem}[{\cite[Thm.~10.81]{Rot09_IntroHomAlg}}]
	\label{thm:Kuenneth}
	Let $A^\bullet$ and $C^\bullet$ be two cochain complexes of $S$-modules.
	Suppose that all terms of $C^\bullet$ and all terms of its coboundary-subcomplex
	are flat $S$-modules.
	
	Then for each $n$ there is a short exact sequence
	
	\begin{center}
	\begin{tikzcd}[column sep=5mm]
		0 
		\ar[r] & \bigoplus_{i+j=n} H^i(A^\bullet) \otimes_S H^j(C^\bullet) 
		\ar[r] & H^n(A^\bullet \otimes_S C^\bullet) \\
		\ar[r] & \bigoplus_{i+j=n+1} \Tor^S_1(H^i(A^\bullet),H^j(C^\bullet)) 
		\ar[r] & 0.
	\end{tikzcd}
	\end{center}
\end{theorem}

\begin{corollary}
	\label{coro:Kuenneth}
	Let $A^\bullet$ and $C^\bullet$ be two cochain complexes of $S$-modules.
	Suppose that all terms of $C^\bullet$ and all terms of its coboundary-subcomplex
	are flat $S$-modules.
	Suppose further that $H^p(C^\bullet)$ is flat for $p < k$,
	$H^p(A^\bullet \otimes_S C^\bullet) = 0$ for $p >k$
	and $\Tor^S_1(H^0(A^\bullet),H^k(C^\bullet)) = 0$.
	
	Then for all $n \neq k$ we have
	\[
		H^n(A^\bullet \otimes_S C^\bullet) \isom \bigoplus_{i+j=n} H^i(A^\bullet) \otimes_S H^j(C^\bullet)
	\]
	and for $n = k$ we have a short exact sequence
	\begin{center}
		\begin{tikzcd}[column sep=5mm]
			0 
			\ar[r] & \bigoplus_{i+j=k} H^i(A^\bullet) \otimes_S H^j(C^\bullet) 
			\ar[r] & H^k(A^\bullet \otimes_S C^\bullet) \\
			\ar[r] & \Tor^S_1(H^1(A^\bullet),H^k(C^\bullet)) 
			\ar[r] & 0.
		\end{tikzcd}
	\end{center}
	
\end{corollary}
\begin{proof}
	By the flatness of $H^p(C^\bullet)$ for $p<k$ and the vanishing of 
	$\Tor^S_1(H^0(A^\bullet),H^k(C^\bullet))$,
	for $n < k$ we have 
	\[
		\bigoplus_{i+j=n+1} \Tor^S_1(H^i(A^\bullet),H^j(C^\bullet)) = 0.
	\]
	So by Theorem \ref{thm:Kuenneth} for $n < k$ we have
	\[
		H^n(A^\bullet \otimes_S C^\bullet) \isom \bigoplus_{i+j=n} H^i(A^\bullet) \otimes_S H^j(C^\bullet)
	\]
	and for $n=k$ we obtain the short exact sequence from Theorem \ref{thm:Kuenneth} by taking the
	flatness of $H^p(C^\bullet)$ for $p < k$ and again the vanishing of $\Tor^S_1(H^0(A^\bullet),H^k(C^\bullet))$
	into account.
	
	Finally, since we assume that $H^p(A^\bullet \otimes_S C^\bullet) = 0$ for $p >k$
	by Theorem \ref{thm:Kuenneth} we also have
	\[
		\bigoplus_{i+j=n} H^i(A^\bullet) \otimes_S H^j(C^\bullet) = 0
	\]
	for all $n > k$.
\end{proof}

The following lemma is helpful to verify the assumptions of
Theorem \ref{thm:Kuenneth} respectively Corollary \ref{coro:Kuenneth}. 
\begin{lemma}
    \label{lem:AssumKuennethFlat}
    Let $C^\bullet$ be a bounded cochain complex consisting of flat $S$-modules
    and assume that $\Tor^S_j(H^i(C^\bullet),M) = 0$ for all $i$, all $j\geq 2$
    and every $S$-module $M$.
    Then the coboundary-subcomplex $B^\bullet$ of $C^\bullet$ also consists of flat $S$-modules.
\end{lemma}
\begin{proof}
    Since $C^\bullet$ is bounded by assumption there is an $m \in \ZZ$ such that $C^i=0$ for all $i>m$.
    In particular $B^i = 0$ for all $i > m$ and is therefore flat.
    Set $H^i := H^i(C^\bullet)$.
    For each $i$ we have the following two canonical short exact sequences
    \begin{align*}
        0 \to B^i &\to Z^i \to H^i \to 0, \\ 
        0 \to Z^{i-1} &\to C^{i-1} \to B^i \to 0. 
    \end{align*}
    
    From these sequences and the associated long exact sequences in $\Tor^S_j(-,M)$, 
    for every $S$-Module $M$ we have the following exact sequences for each $j\geq 0$
    \begin{align}
        \Tor^S_{j+1}(H^i,M) &\to  \Tor^S_j(B^i,M) \to \Tor^S_j(Z^i,M), \tag{1} \label{tors1} \\
        \Tor^S_{j+1}(B^i,M) &\to  \Tor^S_j(Z^{i-1},M) \to \Tor^S_j(C^{i-1},M). \tag{2} \label{tors2} 
    \end{align}
    
    Now we do reverse induction on $i$.
    For $i=m$ we have $Z^m = C^m$, so $\Tor^S_j(Z^m,M) = 0$ for all $j\geq1$. 
    By assumption we also have $\Tor^S_{j+1}(H^m,M) = 0$ for all $j\geq1$ and by 
    (\ref{tors1}) we then have $\Tor^S_j(B^m,M) = 0$ for all $j\geq1$, that is $B^m$ is flat.
    
    Assume that $B^i$ is flat, that is $\Tor^S_j(B^i,M) = 0$ for all $j\geq1$.
    Then by (\ref{tors2}), the flatness of $C^{i-1}$ implies the flatness of $Z^{i-1}$.
    Now, from the first $\Tor$-sequence (\ref{tors1}) (exchanging $i-1$ for $i$) and 
    the vanishing of $\Tor^S_{j+1}(H^{i-1},M)$ for all $j\geq1$ we similarly see
    that $B^{i-1}$ is flat which concludes the induction.
\end{proof}

We note the following property of torsion free $S$-modules.
\begin{lemma}
	\label{lem:Tor1TorFree}
	Let $M$ be a torsion free $S$-module and $p \in S \setminus \{0\}$. Then
	\[
		\Tor^S_1(M,S_p/S) = 0.
	\]
\end{lemma}
\begin{proof}
	Let $Q(S)$ be the quotient field of $S$.
	By \cite[Lem.~7.11]{Rot09_IntroHomAlg} for the torsion free $S$-module $M$ we have
	$\Tor^S_1(M,Q(S)/S) = 0$ which is equivalent to the injectivity of the localization map 
	$f:M = M \otimes S \to M \otimes_S Q(S)$.
	Now, the map $f$ apparently factors through the localization map $g:M \otimes_S S \to M \otimes_S S_p$ 
	and thus $g$ is also injective which in turn implies $\Tor^S_1(M,S_p/S) = 0$.
\end{proof}


\subsection{Sheaves}
\label{ssect:Recoll_Sheaves}

For basics about sheaves and their cohomology we refer to \cite[Ch.~II, III]{Har77_AlgGeom} and \cite[Ch.~5.4, 6.3]{Rot09_IntroHomAlg}.


\subsubsection{Sheaf cohomology}

Let $\Fs$ be a sheaf of abelian groups ($S$-mod\-ules) on the topological space $\Xf$.
The cohomology groups ($S$-modules) of $\Fs$ are defined as the images of the right derived functors
of the global sections functor $\Gamma(\Xf,-):\Sh(\Xf) \to \Ab$ ($\Mod_S$), $\Fs \mapsto \Fs(\Xf)$, that is
\[
    H^n(\Xf,\Fs) 
    	= \Rder^n\Gamma(\Xf,\Fs).
\]

A sheaf $\Gs$ is called \emph{acyclic} if $H^n(\Xf,\Gs) = 0$ for all $n>0$.
An acyclic resolution $\Gs^\bullet$ of $\Fs$ is an exact sequence of sheaves of abelian groups ($S$-modules) on $\Xf$ 
\[
    \Fs \to \Gs^0 \to \Gs^1 \to \Gs^2 \to \ldots
\]
where $\Gs^i$ is acyclic for each $i\geq0$.
Applying the global sections functor to an acyclic resolution yields 
a cochain complex of abelian groups ($S$-modules) $A^\bullet = \Gamma(\Xf,\Gs^\bullet)$ 
whose cohomology computes the sheaf cohomology of $\Fs$: 
\[
    H^n(\Xf,\Fs) \isom H^n(A^\bullet),
\]
cf.\ \cite[Ch.~6]{Rot09_IntroHomAlg}.


\subsubsection{\v{C}ech cohomology}
\label{ssec:Recoll_CechCohom}

Let $\Fs$ be a sheaf of abelian groups on a topological space $\Xf$ and let
$\Uc = \{U_i \mid i \in I\}$ be an open cover of $\Xf$.
Fix a linear order on the index set $I$.

The \emph{\v{C}ech complex} $C^\bullet(\Uc,\Fs)$ is defined as follows.
Set 
\[
    U_{i_0,\ldots,i_n} := U_{i_0}\cap\ldots\cap U_{i_n}.
\]
The terms are
\[
    C^n(\Uc,\Fs) = \prod_{i_0<i_1<\ldots<i_n} \Fs(U_{i_0,\ldots,i_n})
\]
and the coboundary maps are given by
\begin{align*}
    d^n(\alpha)_{i_0<\ldots<i_{n+1}} = 
    \sum_{k=0}^{n+1} (-1)^k \rho|^{U_{i_0,\ldots,\widehat{i_k},\ldots,i_{n+1}}}_{U_{i_0,\ldots,i_{n+1}}}(\alpha_{i_0<\ldots<\widehat{i_k}<\ldots<i_{n+1}}). %
\end{align*}

Then the \v{C}ech cohomology groups (modules) are
\[
    \check{H}^n(\Uc,\Fs) = H^n(C^\bullet(\Uc,\Fs)).
\]
 
The sheaf version of the \v{C}ech complex is defined as follows.
If $j:U_{i_0,\ldots,i_n} \hookrightarrow \Xf$ is the inclusion, define
\[
    \Cs^n(\Uc,\Fs) = \prod_{i_0<i_1<\ldots<i_n} j_*(\Fs|_{U_{i_0,\ldots,i_n}}).
\]
The coboundary maps $d^n$ are defined by the same formula as above.
By \cite[Lem.~III.4.2]{Har77_AlgGeom} we have a resolution $\Fs \to \Cs^\bullet(\Uc,\Fs)$ of $\Fs$.
It is not acyclic in general.

\begin{definition}
    \label{def:LerayCover}
    Suppose we have 
    $H^k(U_{i_0,\ldots,i_n},\Fs|_{U_{i_0,\ldots,i_n}}) = 0$ for all
    finite intersections $U_{i_0,\ldots,i_n}$ of open subsets in $\Uc$ and all $k>0$.
    Then $\Uc$ is called a \emph{Leray cover} for $\Fs$.
\end{definition}

Provided the right setting, the \v{C}ech complex computes sheaf cohomology
by the following classical result due to Leray,
see e.g.\ \cite[Thm.~10.79]{Rot09_IntroHomAlg}. 

\begin{theorem}
    \label{thm:LerayCechCohom}
    If\, $\Uc$ is a Leray cover
    then $\check{H}^n(\Uc,\Fs) \isom H^n(\Xf,\Fs)$ for all $n\geq 0$.
\end{theorem}


\section{Sheaves on Posets}
\label{sec:PosetSheaves}

Let $(P,\leq)$ be a finite poset. The point set of the finite topological space associated to $P$ (also denoted by $P$)
is given by the ground set of $P$.
A topology on $P$ consists of the open subsets which are increasing subsets, i.e.\
$U \subseteq P$ is open if for all $x \in U$ and $y \in P$ with $x\leq y$ we have $y \in U$.
Finite topological spaces of this kind were first studied by Alexandroff \cite{Ale37_DiskRaeume} 
and the topology on $P$ just described is called the \emph{Alexandroff topology}.

\begin{definition}
    \label{def:PrinOpen}
    For $x \in P$ an open subset of the form $U_x := \{y \in P \mid x\leq y\}$ is called \emph{principal}.
    The open cover $\Uc_P := \{U_x \mid x \in P \}$ is called the \emph{principal open cover} of $P$.
\end{definition}

Recall that a poset $P$ can be regarded as a small category with objects the elements of $P$ and
exactly one morphism $x \to y$ for each $x \leq y$.
Every covariant functor $F:P \to \Ab$ ($\Mod_S$) gives rise to a sheaf $\Fs$ of abelian groups ($S$-modules)
on the associated finite topological space as follows. 
The principal open subsets $U_x = \{y \in P \mid x\leq y\}$ form a basis for the topology on $P$.
We define the sections on these principal open sets as
\[
    \Fs(U_x) := F(x)
\]
with restriction maps
\[
    \rho|^{U_x}_{U_y} := F(x\leq y): \Fs(U_x) \to \Fs(U_y).
\]
We leave it to the reader to verify that this extends uniquely to a sheaf $\Fs$ on 
the Alexandroff space $P$, cf.\ \cite{Bac1975_WhitneyNumbers}.
For the stalks $\Fs_x$ we have
\[
    \Fs_x = \varinjlim_{x \in U \underset{\text{open}}{\subseteq} P} \Fs(U) = \Fs(U_x) = F(x).
\]

Recall that a join-semilattice is a poset $L$ where every pair of elements $x,y \in L$
has a least upper bound denoted by $x\vee y = \inf\{z \in L \mid z\geq x$ and $z\geq y\}$.
We note the following.

\begin{remark}
    \label{rem:IntersectLattice}
    Let $L$ be a join-semilattice. Then for two principal open subsets
    $U_x,U_y \subseteq L$ we have
        $U_x \cap U_y = U_{x\vee y}$.
\end{remark}

\bigskip

We now discuss the cohomology of sheaves on posets.
We recall the following fundamental lemma.

\begin{lemma}[{\cite[Lem.~1.1]{Yuz91_LatticeCohom}}]
    \label{lem:PosetCohomMin}
    Let $P$ be a finite poset with a unique minimal element and $\Fs$ a sheaf of abelian groups ($S$-modules) on $P$.
    Then $\Fs$ is acyclic.
\end{lemma}

As a direct consequence of the previous lemma and Remark \ref{rem:IntersectLattice} we obtain the following.

\begin{corollary}
    \label{coro:LerayPoset}
    Let $L$ be a finite  join-semilattice. The principal open cover $\Uc := \{ U_x \mid x \in L \}$ is
    a Leray cover for any sheaf on $L$.
\end{corollary}

\begin{lemma}
    \label{lem:OpenIncl}
    Let $\Fs$ be a sheaf on a finite join-semilattice $L$, $U=U_x \subseteq L$ a principal open subset and $i:U \to L$ the inclusion.
    Then 
    \[
    	H^n(U,\Fs|_U) \isom H^n(L,i_*(\Fs|_U))
    \]
    for each $n \geq 0$, 
    in particular, $i_*(\Fs|_U)$ is acyclic.
\end{lemma}
\begin{proof}
    Note that for a principal open subset $W=U_y \subseteq L$ by Remark \ref{rem:IntersectLattice}
    $W \cap U = U_{x\vee y}$ is also a principal open subset.
    Thus, by \cite[Prop.~III.8.1]{Har77_AlgGeom} and Lemma \ref{lem:PosetCohomMin} 
    for the higher direct image sheaves (that is the right derived functors of the direct image functor) we have
    \[
        \Rder^ki_*(\Fs) (W) = H^k(W\cap U, \Fs|_{W\cap U}) = 0
    \] 
    for $k>0$ and for each $W$ in the principal open cover of $L$.
    Hence $\Rder^ki_*(\Fs) \equiv 0$ for all $k>0$
    and the isomorphism of the cohomology groups follows with \cite[Ex.~III.8.1]{Har77_AlgGeom}.
\end{proof}

\bigskip

Now let $\Xf$ be a topological space and assume that we have a finite open cover $\Uc$ of $\Xf$
which is closed under taking intersections, that is $U\cap U' \in \Uc$ for all $U,U' \in \Uc$.
Suppose further that $\Fs$ is a sheaf of abelian groups ($S$-modules) on $\Xf$.
In this case, we can associate a finite poset $(P_\Uc,\leq)$ to $\Uc$ 
with elements the open sets contained in $\Uc$ and order relation given by reverse inclusion.
In this setting $\Fs$ induces a functor $P_\Uc\to \Ab$ ($\Mod_S$) and hence also a sheaf $\Fs_P$
of abelian groups ($S$-modules) on $P_\Uc$.

\begin{lemma}
    \label{lem:PosetLerayCohom}
    Suppose $\Uc$ is a finite Leray cover for $\Fs$ which is also closed under taking intersections.
    Then 
    \[
        H^n(\Xf,\Fs) \isom H^n(P_\Uc,\Fs_P)
    \]
    for each $n\geq0$.
\end{lemma}
\begin{proof}
    We consider the principal open cover $\Uc'$ of $P_\Uc$.
    Recall the definition of the terms in the \v{C}ech complexes
    of $\Uc$ respectively $\Uc'$. We apparently have
    \[
        C^n(\Uc,\Fs) = C^n(\Uc',\Fs_P)
    \]
    for all $n$.
    Note that $P_\Uc$ is a join-semilattice with $U \vee W = U \cap W$. 
    The cover $\Uc$ is Leray by assumption and the principal cover $\Uc'$ of $P_\Uc$ is
    Leray by Corollary \ref{coro:LerayPoset}. 
    Hence, we obtain the isomorphism of the cohomology groups
    by Theorem \ref{thm:LerayCechCohom}.
\end{proof}

Note that the Alexandroff space of a finite poset $P$ is a noetherian topological space
of dimension $\dim(P) = d$ the maximal length of a maximal chain $x_0 < x_1 < \cdots < x_d$ in $P$.
By Grothendieck's vanishing theorem for the cohomology of sheaves on noetherian spaces \cite[Ch.~3.6]{Gro57_Tohoku} we have the following.

\begin{theorem}
	\label{thm:CohomPosetMax}
	For a finite poset $P$ and for any sheaf $\Fs$ of abelian groups ($S$-modules) on $P$
	we have $H^i(P,\Fs) = 0$ for all $i > \dim(P)$.
\end{theorem}


\section{Sheaves on the intersection lattice}
\label{sec:LASheaves}

Recall that we set $L_0 := ( L(\Ac) \setminus \{ T(\Ac) \} )^{\op}$ so the order relation of $L_0$ is inclusion.
The sheaf on $L_0$ associated to an arrangement considered by Yuzvinsky is defined as follows.
\begin{definition}[{\cite{Yuz91_LatticeCohom}}]
    \label{def:YuzSheaf}
    Denote by $\Ds$ the sheaf associated to the functor $L_0\to \Mod_S, X \mapsto D(\Ac_X)$ given by $\Ds(U_X) := D(\Ac_X)$
    for a principal open subset $U_X = \{ Y \in L_0 \mid X \subseteq Y\} \subseteq L_0$.
    The restriction maps $\Ds(U_Y) \to \Ds(U_X)$ are given by the inclusions $D(\Ac_Y) \subseteq D(\Ac_X)$ for $Y \subseteq X$.
\end{definition}

Note that, since $L(\Ac)$ is a lattice, $L_0$ is a join-semilattice
and 
\[
    X \wedge_{L(\Ac)} Y = X \vee_{L_0} Y.
\]

\bigskip

From now on until the end of this note let $\Xf = \Spec S \setminus \{\mf\}$ 
be the punctured spectrum.

Since we assume the arrangement $\Ac$ to be essential, associated to $\Ac$ is the following affine open cover of $\Xf$
which is in particular Leray for any coherent sheaf on $\Xf$ by a classical result due to Serre \cite{Ser57_CohomAlgVar}.

\begin{definition}
    \label{def:AffCoverA}
    Recall that $Q(X)  = \prod_{H \in \Ac\setminus\Ac_X}\alpha_H = Q(\Ac\setminus \Ac_X)$.
    Define the open cover
    \[
        \Uc_\Ac := \{ U(X) := \Xf \setminus V(Q(X)) \mid X \in L_0 \}
    \]
    of $\Xf$ associated to $\Ac$.
\end{definition}

Note that $U(X) \cap U(Y) = U(X\vee Y)$ for all $X,Y \in L_0$, i.e.\ 
$\Uc_\Ac$ is closed under taking intersections.
We obtain a poset $P_{\Uc_\Ac}$ with order relation given by reverse inclusion
as discussed in Section \ref{sec:PosetSheaves}.

We further have $U(X) \leq U(Y)$ if and only if $X \subseteq Y$, thus the map
\begin{align*}
    L_0 &\to P_{\Uc_\Ac} \\
        X &\mapsto U(X)
\end{align*}
is a poset isomorphism.
If $\Fs$ is a coherent sheaf on $\Xf$, from the general discussion in Section \ref{sec:PosetSheaves}
and the poset isomorphism above, $\Fs$ also defines a sheaf on $L_0$.

\begin{definition}
    \label{def:StrucSheaf}
    (i)
    Let $\Os_\Xf$ be the structure sheaf of the punctured affine space.
    This defines as discussed above a sheaf of $S$-modules on $L_0$ which we denote by $\Os_{L_0}$ with
    \[
    	\Os_{L_0}(U_X) = S_{Q(X)} \quad (X \in L_0).
    \]
    
    (ii)
    Let $\wt{D}|_\Xf$ be the coherent sheaf on $\Xf$ associated to the derivation module $D =D(\Ac)$.
    This induces a sheaf of $S$-modules on $L_0$ denoted by $\wt{D}_{L_0}$ with
    \[
    	\wt{D}_{L_0}(U_X) = D_{Q(X)} \quad (X \in L_0).
    \]
\end{definition}

The tensor product of two sheaves $\Fs, \Gs$ of $S$-modules on a finite poset $P$ is given by
\[
    (\Fs\otimes_S\Gs)(U) := \Fs(U) \otimes_S \Gs(U)
\]
for all $U \subseteq P$ open with restriction maps 
the tensor product of the restriction maps of $\Fs$ and $\Gs$ 
(it suffices to define this for principal open subsets).
A sheaf $\Fs$ of $S$-modules on $L_0$ is flat if $-\otimes_S \Fs$ yields an exact functor.
This is the case if and only if $\Fs_X = \Fs(U_X)$ is a flat $S$-module for all $X \in L_0$.

\begin{remark}
    \label{rem:StrucSheafFlat}
    As $\Os_{L_0}(U_X) = S_{Q(X)}$ is a localization, it is flat for all $X \in L_0$.
    In particular, $\Os_{L_0}$ is a flat sheaf and for the principal open cover $\Uc_{L_0}$ of $L_0$
    all terms in the complex of sheaves $\Cs^\bullet(\Uc_{L_0},\Os_{L_0})$ are flat as they are
    finite direct products of the flat sheaves ${i}_*(\Os_{L_0}|_{U_X})$ where $i:U_X \to L_0$
    is the inclusion of the open subset $U_X$.
\end{remark}

The next result gives the cohomology of the structure sheaf.
\begin{proposition}
    \label{prop:CohomStructSheaf}
    With the notation as above we have
    \[
        H^n(L_0,\Os_{L_0}) \isom H^n(\Xf,\Os_\Xf) = 
            \begin{cases}
                S, \quad &n=0, \\
                S_{x_1x_2\cdots{x_\ell}} / S, \quad &n=\ell-1, \\
                0, \quad &\text{else}.
            \end{cases}
    \]
\end{proposition}
\begin{proof}
    The computation of the cohomology of $\Os_\Xf = \wt{S}|_\Xf$ is a classical result or exercise, 
    see \cite[Thm.~3.8]{Gro67_LocalCohom} and \cite[p.~9]{Hun07_LocalCohomLec}.
    
    The isomorphism $H^n(L_0,\Os_{L_0}) \isom H^n(\Xf,\Os_\Xf)$ follows with
    Lemma \ref{lem:PosetLerayCohom} applied to 
    the open cover $\Uc_\Ac$.  
\end{proof}


\section{Proof of Theorem \ref{thm:IsomCohom}}
\label{sec:Proofs}

Now, all preparations are complete and we put the pieces together to 
prove Theorem \ref{thm:IsomCohom}.

Let $\Ds$, $\wt{D}_{L_0}$ and $\Os_{L_0}$ be the sheaves of $S$-modules on $L_0$ defined in Section \ref{sec:LASheaves}. 
First, we note the following.

\begin{lemma}
    \label{lem:SheafDO}
    We have $\wt{D}_{L_0} = \Ds \otimes_S \Os_{L_0}$.
\end{lemma}
\begin{proof}
This follows immediately from Corollary \ref{coro:LocD} and the definition of the tensor product
of sheaves on $L_0$.
\end{proof}

\begin{proposition}
    \label{prop:ResTensorProd}
    Let $\Os_{L_0}$ and $\Ds$ be as before.
    Then
    \[
        \Cs^\bullet(\Uc_{L_0},\Ds) \otimes_S \Cs^\bullet(\Uc_{L_0},\Os_{L_0})
    \]
    is an acyclic resolution of $\Ds \otimes_S \Os_{L_0}$
    and for all $n\geq0$ we have
    \[
        H^n(L_0,\Ds \otimes \Os_{L_0}) \isom H^n(C^\bullet(\Uc_{L_0},\Ds) \otimes_S C^\bullet(\Uc_{L_0},\Os_{L_0}) ).
    \]   
\end{proposition}
\begin{proof}
    Since $\Uc_{L_0}$ is a finite cover, both of the complexes $\Cs^\bullet(\Uc_{L_0},\Ds)$, $\Cs^\bullet(\Uc_{L_0},\Os_{L_0})$
    are bounded and so is the bicomplex of their tensor product. 
    By Remark \ref{rem:StrucSheafFlat} all terms of $\Cs^\bullet(\Uc_{L_0},\Os_{L_0})$
    are flat and the exactness of the complex $\Cs^\bullet(\Uc_{L_0},\Ds) \otimes_S \Cs^\bullet(\Uc_{L_0},\Os_{L_0})$ 
    follows with Lemma \ref{lem:AcycAss}.
    
    It remains to show, that all terms in the tensor product complex are acyclic sheaves.
    
    All terms of $\Cs^\bullet(\Uc_{L_0},\Ds) \otimes_S \Cs^\bullet(\Uc_{L_0},\Os_{L_0})$ are finite direct products
    of sheaves of the form
    \[
        \Cs_{X,Y} := {i_X}_*(\Ds|_{U_X}) \otimes_S {i_Y}_*(\Os_{L_0}|_{U_Y})
    \]
    for $X, Y \in L_0$ and inclusion maps $i_X:U_X \to L_0$, $i_Y:U_Y \to L_0$.
    For $Z \in L_0$ we have
    \begin{align*}
        \Cs_{X,Y}(U_Z) 
        &= {i_X}_*(\Ds|_{U_X})(U_Z) \otimes_S {i_Y}_*(\Os_{L_0}|_{U_Y})(U_Z) \\
        &= \Ds(U_{X\vee Z}) \otimes_S \Os_{L_0}(U_{Y\vee Z}) \\
        &= D(\Ac_{X\wedge_{L(\Ac)}Z})_{Q(Y\wedge_{L(\Ac)}Z)} \\
        &=D(\Ac_{X\wedge_{L(\Ac)}Y\wedge_{L(\Ac)}Z})_{Q(Y\wedge_{L(\Ac)}Z)} \\
        &=D(\Ac_X)_{Q(Y\wedge_{L(\Ac)}Z)},
    \end{align*}
    where the last two equalities hold thanks to Lemma \ref{lem:LocD}.
    Hence 
    \[
        \Cs_{X,Y} = {i_Y}_*(\wt{D(\Ac_X)}|_{U_Y})
    \]
    where $\wt{D(\Ac_X)}$ is the sheaf associated to the derivation module of the localization $\Ac_X$
    as in Definition \ref{def:StrucSheaf}.
    As a direct image sheaf of an inclusion of a principal open subset it is acyclic 
    by Lemma \ref{lem:OpenIncl} and so are all the terms of $\Cs^\bullet(\Uc_{L_0},\Ds) \otimes_S \Cs^\bullet(\Uc_{L_0},\Os_{L_0})$.
    Consequently, $\Cs^\bullet(\Uc_{L_0},\Ds) \otimes_S \Cs^\bullet(\Uc_{L_0},\Os_{L_0})$ is an acyclic resolution of $\Ds \otimes_S \Os_{L_0}$
    as desired.
    Finally, note that by the definition of the tensor product of sheaves on $L_0$ for the global sections we have
    \[
        \Gamma(L_0,\Cs^\bullet(\Uc_{L_0},\Ds) \otimes_S \Cs^\bullet(\Uc_{L_0},\Os_{L_0})) = C^\bullet(\Uc_{L_0},\Ds) \otimes_S C^\bullet(\Uc_{L_0},\Os_{L_0})
    \]
    and so the cohomology of the tensor product of the two \v{C}ech complexes computes the cohomology
    of $\Ds \otimes_S \Os_{L_0}$.
\end{proof}

It remains to verify the assumptions of the K\"unneth formula (Theorem \ref{thm:Kuenneth} resp.\ Corollary \ref{coro:Kuenneth})
for the \v{C}ech complex of $\Os_{L_0}$. 
This is done by the following proposition.

\begin{proposition}
    \label{prop:AssKuenneth}
    All terms in $C^\bullet(\Uc_{L_0},\Os_{L_0})$, all terms of its co\-boun\-dary-subcomplex and
    $H^p(C^\bullet(\Uc_{L_0},\Os_{L_0}) )$ ($p < \ell-1$) are flat $S$-modules.
\end{proposition}
\begin{proof}
    First note that the complex $C^\bullet(\Uc_{L_0},\Os_{L_0})$ is bounded since $\Uc_{L_0}$ is a finite cover.
    By Remark \ref{rem:StrucSheafFlat} all $C^i(\Uc_{L_0},\Os_{L_0})$ are flat
    and note that since $\Uc_{L_0}$ is a Leray cover we have $H^p(L_0,\Os_{L_0}) \isom H^p(C^\bullet(\Uc_{L_0},\Os_{L_0}))$. 
    
    Recall that $\Tor^S_j(A/B,M) = 0$ for all $j\geq2$ and every $S$-module $M$
    provided $A$ and $B$ are both flat $S$-modules.
    So by Proposition \ref{prop:CohomStructSheaf} we have
    \[
        \Tor^S_j( H^p(C^\bullet(\Uc_{L_0},\Os_{L_0})),M) = 0
    \]
    for each $j\geq2$ and every $S$-module $M$.
    Consequently, the complex $C^\bullet(\Uc_{L_0},\Os_{L_0})$ satisfies the assumptions of 
    Lemma \ref{lem:AssumKuennethFlat} and so all terms of the coboundary-subcomplex
    of $C^\bullet(\Uc_{L_0},\Os_{L_0})$ are flat.
    
    Finally, once more by Proposition \ref{prop:CohomStructSheaf} the
    modules $H^p(C^\bullet(\Uc_{L_0},\Os_{L_0}))$ are flat $S$-modules for $p < \ell-1$.
\end{proof}

\begin{proposition}
    \label{prop:TensorProdDO}
    For $n \neq \ell-1$ we have 
    \[
        H^n(L_0,\wt{D}_{L_0}) \isom \bigoplus_{i+j=n} H^i(L_0,\Ds) \otimes_S H^j(L_0,\Os_{L_0})
    \]
    and for $n=\ell-1$ we have a short exact sequence
    \begin{center}
    	\begin{tikzcd}[column sep=5mm]
    		0 
    		\ar[r] & \bigoplus_{i+j={\ell-1}} H^i(L_0,\Ds) \otimes_S H^j(L_0,\Os_{L_0})
    		\ar[r] & H^{\ell-1}(L_0,\wt{D}_{L_0}) \\
    		\ar[r] & \Tor^S_1(H^1(L_0,\Ds),H^{\ell-1}(L_0,\Os_{L_0}))
    		\ar[r] & 0.
    	\end{tikzcd}
    \end{center}
\end{proposition}
\begin{proof}
    By Lemma \ref{lem:SheafDO}, we readily get $H^n(L_0,\wt{D}_{L_0}) \isom H^n(L_0,\Ds \otimes_S \Os_{L_0})$.
    Owing to Proposition \ref{prop:ResTensorProd} we have 
    \[
    	H^n(L_0,\Ds \otimes_S \Os_{L_0}) \isom  H^n(C^\bullet(\Uc_{L_0},\Ds) \otimes_S C^\bullet(\Uc_{L_0},\Os_{L_0}) ).
    \]
    By Proposition \ref{prop:AssKuenneth} the coboundary-subcomplex of $C^\bullet(\Uc_{L_0},\Os_{L_0})$ is flat
    and by Proposition \ref{prop:CohomStructSheaf} $H^p(C^\bullet(\Uc_{L_0},\Os_{L_0}))$ is flat for all $p< \ell-1$.
    Further, as the dimension of $L_0$ is $\ell-1$, by Theorem \ref{thm:CohomPosetMax} we have 
    \[
    	H^p(C^\bullet(\Uc_{L_0},\Ds) \otimes_S C^\bullet(\Uc_{L_0},\Os_{L_0}) ) = H^p(L_0,\Ds \otimes_S \Os_{L_0}) = 0
    \]
    for $p > \ell-1$.
    Moreover, $H^0(C^\bullet(\Uc_{L_0},\Ds)) \isom H^0(L_0,\Ds) \isom D(\Ac)$ 
    is a reflexive $S$-module (cf.\ \cite[p.~268]{Saito80_LogForms}) 
    and as such it is torsion free in particular. 
    Since $H^{\ell-1}(C^\bullet(\Uc_{L_0},\Os_{L_0})) \isom H^{\ell-1}(L_0,\Os_{L_0}) \isom S_{x_1\cdots x_\ell}/S$ 
    by Proposition \ref{prop:CohomStructSheaf}
    we have 
    \[
    	\Tor^S_1(H^0(C^\bullet(\Uc_{L_0},\Ds)), H^{\ell-1}(C^\bullet(\Uc_{L_0},\Os_{L_0}))) = 0
    \]
    by Lemma \ref{lem:Tor1TorFree}.
   	Hence, we can applying  Corollary \ref{coro:Kuenneth} to the two complexes $C^\bullet(\Uc_{L_0},\Ds)$ and $C^\bullet(\Uc_{L_0},\Os_{L_0})$
   	which concludes the proof.
\end{proof}

By Lemma \ref{lem:PosetLerayCohom} we have isomorphisms 
$H^n(L_0,\wt{D}_{L_0}) \isom H^n(\Xf,\wt{D}|_\Xf)$ and
$H^n(L_0,\Os_{L_0}) \isom H^n(\Xf,\Os_{\Xf})$ for all $n\geq 0$.
Recall that we have $H^0(\Xf,\Os_\Xf) \isom S$ by Proposition \ref{prop:CohomStructSheaf}.
This yields our main theorem.

\begin{corollary}[{Theorem \ref{thm:IsomCohom}}]
	For all $n\neq \ell-1$ we have:
	\[
	H^n(\Xf,\wt{D}|_\Xf) \isom \bigoplus_{i+j=n} H^i(L_0,\Ds) \otimes_S H^j(\Xf,\Os_\Xf)
	\]
	and for $n=\ell-1$ we have a short exact sequence
	\begin{center}
		\begin{tikzcd}[column sep=5mm]
			0 
			\ar[r] & \bigoplus_{i+j={\ell-1}} H^i(L_0,\Ds) \otimes_S H^j(\Xf,\Os_\Xf)
			\ar[r] & H^{\ell-1}(\Xf,\wt{D}|_\Xf) \\
			\ar[r] & \Tor^S_1(H^1(L_0,\Ds),H^{\ell-1}(\Xf,\Os_\Xf))
			\ar[r] & 0.
		\end{tikzcd}
	\end{center}
	
	In particular, $H^n(\Xf,\wt{D}|_\Xf) \isom H^n(L_0,\Ds)$ for $n < \ell-1$.
\end{corollary}


\section{Concluding remarks}
	\label{sec:ConclRmks}

In view of Corollary \ref{cor:CharFreeness}, we may reformulate Terao's conjecture as follows:

\begin{conjecture}
	The vanishing of the lattice sheaf cohomology groups $H^n(L_0,\Ds)$ for $0<n<\ell-1$
	does only depend on the poset $L_0$.		
\end{conjecture}

In his recent work \cite{Abe2020_pdD}, Abe studies the behavior of the projective dimension
under addition-deletion operations. 
He also states the problem \cite[Prob.~6.3(3)]{Abe2020_pdD}
whether indeed even the projective dimension of the derivation module is combinatorial.
This generalization of Terao's conjecture can be reformulated with Theorem \ref{thm:pdD} as follows.

\begin{problem}
	Does the vanishing of the lattice sheaf cohomology groups $H^n(L_0,\Ds)$ (for arbitrary $n$)
	only depend on the combinatorics of the arrangement, i.e.\ on the poset $L_0$?	
\end{problem}

Note that Lemma \ref{lem:LocD} easily generalizes to the case of multi-\-arran\-ge\-ments, i.e.
arrangements equipped with a multiplicity function $\mathbf{m}:\Ac \to \ZZ _{\geq 0}$ which were
introduced by Ziegler \cite{Ziegler1989_MultiArr}.
Moreover, the other essential properties of $D(\Ac)$ used in the proof of Theorem \ref{thm:IsomCohom} 
hold more generally for the module of multi-derivations $D(\Ac,\mathbf{m})$.
Hence, all the main results of Section \ref{sec:Intro} extend to multi-arrangements 
essentially with the same proofs.

\bigskip

In view of the important results about freeness preserved under various addition, deletion and 
restriction operations, cf.\ \cite{Terao1980_FreeI}, \cite{Ziegler1989_MultiArr}, \cite{Yos04_CharaktFree}, \cite{Abe16_DivFree},
thinking of the long exact sequence in cohomology obtained from a short exact sequence of sheaves,
we note the following natural problem.

\begin{problem}
	Are there short exact sequences relating the sheaves $\Ds, \Ds', \Ds''$, $(\Ds^H,\mathbf{m}^H)$
	where $\Ds', \Ds'', (\Ds^H,\mathbf{m}^H)$ are the sheaves of a deletion, restriction or Ziegler-restriction
	of the given arrangement, respectively?	
\end{problem}

Finally, we suspect that Yuzvinsky's celebrated theorem \cite[Thm.~3.3]{Yuz93_FreeLocFreeL}
stating that the subset formed by free arrangements in the moduli space of all arrangements with a given intersection
lattice is Zariski open can be generalized with our results:
arrangements with projective dimesion greater or equal to $p$ constitute a Zariski closed
subset in the moduli space of all arrangements with a given intersection lattice, or
equivalently, arrangements with projective dimension less or equal to $p$ form a Zariski open
subset.


\section*{Acknowledgments}

The author is grateful to Wassilij Gnedin for several helpful discussions and 
for suggesting Lemma \ref{lem:AssumKuennethFlat} to verify in our situation
the assumptions for Theorem \ref{thm:Kuenneth} respectively Corollary \ref{coro:Kuenneth}.
He  would also like to thank Gerhard R\"ohrle
for helpful comments on an earlier draft of the manuscript.



\providecommand{\bysame}{\leavevmode\hbox to3em{\hrulefill}\thinspace}
\providecommand{\MR}{\relax\ifhmode\unskip\space\fi MR }
\providecommand{\MRhref}[2]{%
	\href{http://www.ams.org/mathscinet-getitem?mr=#1}{#2}
}
\providecommand{\href}[2]{#2}

\end{document}